\documentclass[11pt,letterpaper]{amsart}

\usepackage{hyperref}

\numberwithin{equation}{section}
\theoremstyle{plain}
\newtheorem{theorem}{Theorem}[section]

\newtheorem{lemma}[theorem]{Lemma}

\newtheorem{corollary}[theorem]{Corollary}

\theoremstyle{definition}

\newtheorem{remark}[theorem]{Remark}

\newcommand{\Rmnum}[1]{\expandafter\@slowromancap\romannumeral #1@}

\newcommand{\ud}{\mathrm{d}}

\allowdisplaybreaks

\address{Mingxiang Li, Department  of Mathematics \& Institue of Mathematical Sciences, Chinese University of Hong Kong, Shatin, NT, Hong Kong}
\email{mingxiangli@cuhk.edu.hk}
\address{Juncheng Wei,  Department of Mathematics, Chinese University of Hong Kong, Shatin, NT, Hong Kong}
\email{wei@math.cuhk.edu.hk}
\thanks{ This research  is partially supported  by   Hong Kong General Research Fund "New frontiers in singular limits of nonlinear partial differential equations". We thank Professor Jeffrey Case for reading the manuscript and suggestions.}

\begin{document}
	\title[Liouville type theorems]{A Remark On  Case-Gursky-V\'etois identity and its applications}
	\author{ Mingxiang Li,  Juncheng Wei}
	\date{}
	\maketitle
	\begin{abstract}
		Based on  the works of Gursky (CMP, 1997),  V\'etois (Potential Anal., 2023)  and  Case (Crelle's journal, 2024), we  make use of  an Obata type formula established in these works  to obtain  some Liouville type theorems on conformally Einstein manifolds.  In particular, we solve  Hang-Yang conjecture (IMRN, 2020)  via an Obata-type argument and obtain optimal perturbation.
	\end{abstract}
	
	\section{Introduction}

	Given a Riemannian  manifold $(M^n,g)$ with $n\geq 3$, it is well known that
	Branson's Q-curvature \cite{Branson}  is defined by
	\begin{equation}\label{def of Q-curvature}
		Q_g:=-\frac{1}{2(n-1)}\Delta_g R_g-\frac{2}{(n-2)^2}|E_g|^2_g+\frac{n^2-4}{8n(n-1)^2}R_g^2
	\end{equation}
	where $R_g$ is the scalar curvature  and $E_g$ is the trace-free Ricci tensor  defined by   $E_g:=Ric_g-\frac{R_g}{n}g$.  Now, we give some notations and a brief overview of Q-curvature in conformal geometry.  The Schoten tensor is given by
	$$A_g=\frac{1}{n-2}\left(Ric_g-\frac{R_g}{2(n-1)}g\right)$$
	and $\sigma_k(A_g)$ denote the k-th symmetric functions of the eigenvalues of $A_g$. Then we can rewrite Q-curvature in \eqref{def of Q-curvature} as follows
	\begin{equation}\label{Q-curvature def}
		Q_g=-\Delta_g\sigma_1(A_g)+4\sigma_2(A_g)+\frac{n-4}{2}\sigma_1^2(A_g)
	\end{equation}
	The remarkable Paneitz operator  \cite{Paneitz} is given by
	\begin{equation}\label{Paneitz operator}
		P_g\varphi=\Delta^2_g\varphi+div_g\left\{(4A_g-(n-2)\sigma_1(A_g)g)(\nabla \varphi ,\cdot)\right\}+\frac{n-4}{2}Q_g\varphi.
	\end{equation}
	For $n\geq 3$ and $n\not=4$, the Q-curvature $Q_{\tilde g}$ of the  conformal metric $\tilde g=	\varphi^{\frac{4}{n-4}}g$ satisfies
	\begin{equation}\label{conformal Q-curvature n not 4}
		P_g\varphi=\frac{n-4}{2}Q_{\tilde g}\varphi^{\frac{n+4}{n-4}}.
	\end{equation}
	For  $n=4$, the Q-curvature $Q_{\tilde g}$ of the  conformal metric $\tilde g=	e^{2\varphi}g$ satisfies
	\begin{equation}\label{conformal Q-curvature for n=4}
		P_g\varphi +Q_g=Q_{\tilde g}e^{4\varphi}.
	\end{equation}
	
	Similar to the Yamabe problem, it is crucial to identify a conformal metric such that the Q-curvature is constant, which is equivalent to solving the fourth-order nonlinear equations \eqref{conformal Q-curvature n not 4} and \eqref{conformal Q-curvature for n=4} with $Q_{\tilde g}\equiv C$ for some constant $C$. For $ n=4$, significant advancements were made by Chang and Yang \cite{Chang-Yang 95 Ann} and Djadli and Malchiodi \cite{Dj-Ma 08 Ann}. In dimensions $n\geq 5$,  Gursky and Malchiodi \cite{G-M} and Hang and Yang \cite{Hang-Yang CPAM for ngeq 5} established existence results under suitable conditions related to scalar curvature. For $n=3$, a similar result was achieved by Hang and Yang \cite{Hang-Yang}.

	The well known Obata theorem (\cite{Obata 1}, \cite{Obata 2}) states that for an  Einstein manifold $(M^n,g_0)$ with $n\geq3$, if the scalar curvature of conformal metric $g=u^2g_0$ is a constant, $g$ must also  be Einstein. Furthermore,  if    $(M,g_0)$ is not conformally equivalent to round sphere, $u$ must be a positive constant.

	For the Einstein manifold $(M^n,g_0)$ with scalar curvature $R_0$, the Paneitz operator can be succinctly expressed as follows (See \cite{Gover} for more details)
	\begin{equation}\label{Einstein paneitz}
		P_{g_0}\varphi=\Delta_{g_0}^2\varphi-\frac{n^2-2n-4}{2n(n-1)}R_{0}\Delta_{g_0}\varphi+\frac{n-4}{2}Q_0\varphi
	\end{equation}
	where the   Q-curvature $Q_0$ is given by
	\begin{equation}\label{Einstein Q_0}
		Q_0=\frac{n^2-4}{8n(n-1)^2}R_0^2.
	\end{equation}

	It is natural to ask whether similar Obata type theorem holds for Q-curvature. Firstly,  V\'etois \cite{Vetois} made use of Bochner–Lichnerowicz–Weitzenbock formula  and a important Lemma \ref{GM lemma} established by Gursky and Malchiodi   \cite{G-M} to deal with the equations \eqref{conformal Q-curvature n not 4} and \eqref{conformal Q-curvature for n=4}. Then he establishes the following  Obata type theorem.

	\begin{theorem}\label{theorem of Vetois}(V\'etois' theorem in \cite{Vetois})
		Suppose that $(M^n,g_0)$ where $n\geq 3$ is a compact  Einstein manifold with non-negative  scalar curvature $R_0$. Consider a conformal metric
		$g=u^2g_0$ where $u>0$. Suppose that the Q-curvature of conformal metric  $g$ is constant.  Then $g$ is Einstein.  Furthermore,  if    $(M,g_0)$ is not conformally equivalent to round sphere, $u$ must be a positive constant.
	\end{theorem}
	\begin{remark}
		In  original version of V\'etois's theorem(Theorem 1.1 in \cite{Vetois}), he didn't cover the case $R_0=0$. In fact, by integrating \eqref{conformal Q-curvature n not 4} and \eqref{conformal Q-curvature for n=4} using the representation of  \eqref{Einstein paneitz}, it is not hard to include this case.
	\end{remark}

	Recently,   Case \cite{Case} consider a more general Obata-V\'etois type theorem.  For a compact manifold $(M^n,g)$, he introduced  the following mixed curvature
	\begin{equation}\label{I_a curvature }
		I_a(g)=Q_g+a\sigma_2(A_g)
	\end{equation}
	where $Q_g$ and $A_g$ are the Q-curvature and Schouten tensor respectively given before.

		\begin{theorem}\label{theorem case}(Case's theorem in \cite{Case})
		Suppose that $(M^n,g_0)$ is an Einstein manifolds with the scalar curvature $R_0\geq 0$ and $n\geq 3$. Consider the conformal metric
		$g=u^2g_0$ where $u>0$ satisfying $R_g\geq 0$ and $I_a(g)$ is a constant.  If the constant $a$ satisfies
			\begin{equation}\label{a range}
			\frac{(n-2)(n-4)-n\sqrt{n^2+4n}}{2(n-1)}\leq a\leq \frac{(n-2)(n-4)+n\sqrt{n^2+4n}}{2(n-1)},
		\end{equation}
		 then $g$ is Einstein.
	\end{theorem}
		\begin{remark}
		In original version of Case's theorem,  he considered the following interval $$B_1:=[\frac{n^2-7n+8-\sqrt{n^4+2n^3-3n^2}}{2(n-1)}, \frac{n^2-7n+8+\sqrt{n^4+2n^3-3n^2}}{2(n-1)}].$$
		It is easy to check that
		$$B_1\subset [\frac{(n-2)(n-4)-n\sqrt{n^2+4n}}{2(n-1)},\frac{(n-2)(n-4)+n\sqrt{n^2+4n}}{2(n-1)}].$$
		In fact, with the help of the sharp inequality Lemma \ref{lem: A}, we slightly extend the range of the constant  $a$.
		However, we still  do not know how to get the optimal range which is a very interesting question.
	\end{remark}

	To prove Theorem \ref{theorem case}, Case established a remarkable identity (See Lemma 3.2 in \cite{Case}) as follows
	\begin{align}
			0=&\frac{1}{2}\int_Mu|\nabla R_g|^2\ud v_g+\frac{n-(n-1)(a+4)}{n-2}\int_M E_g(\nabla R_g,\nabla u)\ud v_g\label{Case's identity}\\
		&+\frac{n(n-1)^2(a+4)}{2(n-2)^2}\int_M|E_g|^2u^{-1}|\nabla u|^2\ud v_g\nonumber \\
		&+\frac{(n-1)(a+4)+2n^2-4n}{2(n-2)^2}\int_Mu|E_g|^2R_g\ud v_g\nonumber\\
		&+\frac{(n-1)(a+4) R_0}{2(n-2)^2}\int_M|E_g|^2u^{-1}\ud v_g\nonumber
	\end{align}
	under the assumption that $I_a(g)$ is a constant.   Such formula is also obtained by Gursky(See the equation (1.13) in \cite{Gur 97 CMP}) on $\mathbb{S}^4$ and V\'etois (See equation (1.4) in \cite{Vetois}) for $a=0$.

In \cite{Vetois}, V\'etois established a more general Liouville-type theorem by considering the following equations
\begin{align}
	&P_{g_0}\varphi =\varphi^p, \quad n\not=4, \; p\leq\frac{n+4}{n-4},\label{n not4}\\
	&P_{g_0}\varphi=e^{p\varphi},\quad  n=4, \; p\leq 4,\label{n=4}
\end{align}
on compact Einstein manifolds. He showed that the positive solutions to \eqref{n not4} and smooth solutions to \eqref{n=4} must be constant if $(M^n,g)$ is not conformally equivalent to the  round sphere.

Based on Case's formula \eqref{Case's identity} and Gursky's work \cite{Gur 97 CMP}, we observe that we can add the term $I_g(g)$ into the formula  \eqref{Case's identity} without assuming $I_a(g)$ is a constant. Such formula will play an important role in the proof of Liouville type theorems.
On the whole, the  Case-Gursky-V\'etois  formula can be written as follows
\begin{align}
		&2(n-1)^2\int_M\langle \nabla I_a(g),\nabla u\rangle_g\ud v_g\label{generalize Case formula}\\
		=&\frac{1}{2}\int_Mu|\nabla R_g|^2\ud v_g+\frac{n-(n-1)(a+4)}{n-2}\int_M E_g(\nabla R_g,\nabla u)\ud v_g\nonumber \\
		&+\frac{n(n-1)^2(a+4)}{2(n-2)^2}\int_M|E_g|^2u^{-1}|\nabla u|^2\ud v_g\nonumber\\
		&+\frac{(n-1)(a+4)+2n^2-4n}{2(n-2)^2}\int_Mu|E_g|^2R_g\ud v_g\nonumber\\
		&+\frac{(n-1)(a+4) R_0}{2(n-2)^2}\int_M|E_g|^2u^{-1}\ud v_g.\nonumber
	\end{align}
	With help of such  formula, we are able to streamline   the proof of the Liouville-type theorems of V\'etois(See Theorem 2.1 and Theorem 2.2 in \cite{Vetois}). Besides,  we can  generalize the Liouville-type results from \cite{BVV}, \cite{Brezis-Li}, and \cite{Gidas-Sprcuk CPAM} related to second-order nonlinear equations to fourth-order cases by considering  the perturbation  of linear term of Paneitz operator.  Since the situations differ slightly for $n\geq 5, n=4$ and $ n=3$, we establish the results for each case separately.
	
	\begin{theorem}\label{thm: neq 5}
		Suppose that $(M^n,g_0)$ where  $n\geq 5$  is a compact  Einstein manifold with positive  scalar curvature $R_0$.  Consider the positive solution $\varphi$ to  the equation
		\begin{equation}\label{nonlinear for n geq 5}
			P_{g_0}\varphi-\varepsilon\varphi=\varphi^{p}
		\end{equation}
		where $0\leq \varepsilon<\frac{n-4}{2}Q_0$ and $p\leq \frac{n+4}{n-4}$. If $(M^n,g_0)$ is conformally equivalent to round sphere, we additionally assume that $\varepsilon+\frac{n+4}{n-4}-p>0$. Then $\varphi$ must be a  constant.
	\end{theorem}
	\begin{remark} By integrating the equation  \eqref{nonlinear for n geq 5} over $(M,g_0)$, it is easy to see that the condition
		$\varepsilon<\frac{n-4}{2}Q_0$ is necessary for $\varphi>0.$
	\end{remark}
	
	\begin{theorem}\label{thm: n=4}
		Suppose that $(M^4,g_0)$  is a compact  Einstein manifold with positive  scalar curvature $R_0$.  Consider the  solution $\varphi$ to  the equation
		\begin{equation}\label{nonlinear for n =4}
			P_{g_0}\varphi+Q_0-\varepsilon=e^{p\varphi}
		\end{equation}
		where $0\leq  \varepsilon<  Q_0$ and $p\leq 4$. If $(M^4,g_0)$ is conformally equivalent to round sphere, we additionally assume that $\varepsilon+4-p>0$. Then $\varphi$ must be a  constant.
	\end{theorem}

	\begin{theorem}\label{thm: n=3}
		Suppose that $(M^3,g_0)$  is a compact  Einstein manifold with positive  scalar curvature $R_0$.  Consider the  positive solution $\varphi$ to  the equation
		\begin{equation}\label{nonlinear for n =3}
			P_{g_0}\varphi+\varepsilon\varphi =-\varphi^{p}
		\end{equation}
		where $0\leq  \varepsilon<  \frac{Q_0}{2}$ and $p\geq -7$. If $(M^3,g_0)$ is conformally equivalent to round sphere, we additionally assume that $\varepsilon+7+p>0$. Then $\varphi$ must be a  constant.
	\end{theorem}
	
	This result recovers Conjecture 1.1 of Hang and Yang \cite{Hang-Yang 20 IMRN}, which addresses the standard sphere $\mathbb{S}^3$ via an Obata-type proof.  Firstly,   Zhang \cite{Zhang SH}  solved this conjecture by transforming the equation on the sphere into Euclidean space and applying the moving plane method.  Later, Hyder and Ngô \cite{Hyder-Ngo}  generalized this theorem to higher order cases by using similar moving plane method.
	
	We should point out that  their approaches  can only handle small values of $\varepsilon$ since their proofs need a compactness theorem.  When    $p=-7$ and $(M^3,g_0)$ is the   round sphere $\mathbb{S}^3$, Theorem \ref{thm: n=3} establishes  this Liouville-type theorem under the optimal range  $0<\varepsilon<\frac{Q_0}{2}$ via an Obata type argument.

	In fact, with the  help of the strong maximum principle established by Gursky and Malchiodi (Theorem 2.2 in \cite{G-M}) and Hang-Yang (Proposition 2.1 in  \cite{Hang-Yang}), we are able to establish a more general result as below.
	
	\begin{theorem}\label{thm: f(u) n not 4}
		Suppose that $(M^n,g_0)$ where  $n\geq 3$ and $n\not=4$  is a compact  Einstein manifold with positive  scalar curvature $R_0$.
		Consider the solution $\varphi\in C^4(M^n,g_0)$ to  the following equation
		$$P_{g_0}\varphi =\frac{n-4}{2}f(\varphi)$$
		where $f(t)\geq 0$  for all $t\in\mathbb{R}$ is a smooth function  satisfying
		\begin{equation}\label{condition for f(t)}
			\frac{n-4}{2}	\partial_t\left(t^{-\frac{n+4}{n-4}}f(t)\right)\leq 0, \quad \forall \; t>0.
		\end{equation}
		If $(M^n,g_0)$ is conformally equivalent to the round sphere, we additionally assume that  the inequality in \eqref{condition for f(t)} is strict. Then $\varphi$ must be a constant.
	\end{theorem}
	For $n=4$, a similar result holds.
	\begin{theorem}\label{thm: f(u) n =4}
		Suppose that $(M^4,g_0)$ is a four-dimensional compact  Einstein manifold with positive  scalar curvature $R_0$.
		Consider the solution $\varphi\in C^4(M^4,g_0)$ to the  following equation
		$$P_{g_0}\varphi+ Q_0 =f(e^{\varphi})$$
		where $f(t)\geq 0$  for all $t>0$ is a smooth function  satisfying
		\begin{equation}\label{condition for f(t) n=4}
			\partial_t\left(t^{-4}f(t)\right)\leq 0, \quad \forall \; t>0.
		\end{equation}
		If $(M^4,g_0)$ is conformally equivalent to the round sphere, we additionally assume that  the inequality in \eqref{condition for f(t) n=4} is strict. Then $\varphi$ must be a constant.
	\end{theorem}

	This paper is organized as follows. In Section \ref{sect:2}, we prove  the  Case-Gursky-V\'etois  formula  by following the argument of Gursky \cite{Gur 97 CMP} and Case \cite{Case}. With the help of such identity,  we establish two  Case-Gursky-V\'etois   rigidity inequalities. Finally, in Section \ref{sec:4}, with the help of rigidity inequalities,  we give the proofs of Theorem \ref{thm: neq 5}, Theorem \ref{thm: n=4}, Theorem \ref{thm: n=3}, Theorem \ref{thm: f(u) n not 4} and Theorem \ref{thm: f(u) n =4}.

	\section{Case-Gursky-V\'etois  identity on conformal Einstein manifolds}\label{sect:2}

	 Before doing so, we introduce some notations for later use.
	Notice that
	$$\sigma_2(A_g)=-\frac{|E_g|_g^2}{2(n-2)^2}+\frac{1}{8(n-1)n}R^2_g.$$
	Then, with help of  \eqref{def of Q-curvature} and \eqref{I_a curvature }, one has
	\begin{equation}\label{I_a curvature repre}
		2(n-1)I_a(g)=-\Delta_g R_g-\alpha_1|E_g|_g^2+\alpha_2R_g^2
	\end{equation}
	where there constants $\alpha_1$ and $\alpha_2$ are defined as follows
	\begin{equation}\label{def alpha_1,2}
		\alpha_1=\frac{(n-1)(4+a)}{(n-2)^2}, \quad \alpha_2=\frac{(n-1)a+n^2-4}{4(n-1)n}.
	\end{equation}

	Now,  we are going to give give the proof of Case-Gursky-V\'etois identity \eqref{generalize Case formula}. In the proof of this  formula,  we basically follow Case \cite{Case}, using idea of Gursky \cite{Gur 97 CMP}.   The key observation  during Gursky's  proof (1997, CMP, Page 660) is that integrating
	$u\langle E_g, \nabla^2R_g\rangle_g $ over $(M,g)$, then insert the representations of $\Delta_g u$ and $\Delta_g R_g$.  Here,  we need’t assume that $I_a(g)$ is a constant. Then we obtain an integral identity (Theorem 2.1).  We should point out that
	this  identity can  also be obtained by inserting the tensor $T$ defined in Page 4 of \cite{Case}  into  the equation (1.6) without assuming $I_a(g)$ is a constant.

	\begin{theorem}\label{thm: identity}(Case-Gursky-V\'etois identity)
		Suppose that $(M^n,g_0)$ is an Einstein manifold with constant scalar curvature $R_0$ and $n\geq 3$. Consider a conformal metric
		$g=u^2g_0$ where $u>0$.  Then there holds
		\begin{align*}
		&2(n-1)^2\int_M\langle \nabla I_a(g),\nabla u\rangle_g\ud v_g\\
		=&\frac{1}{2}\int_Mu|\nabla R_g|^2\ud v_g+\frac{n-(n-1)(a+4)}{n-2}\int_M E_g(\nabla R_g,\nabla u)\ud v_g\\
		&+\frac{n(n-1)^2(a+4)}{2(n-2)^2}\int_M|E_g|^2u^{-1}|\nabla u|^2\ud v_g\\
		&+\frac{(n-1)(a+4)+2n^2-4n}{2(n-2)^2}\int_Mu|E_g|^2R_g\ud v_g\\
		&+\frac{(n-1)(a+4) R_0}{2(n-2)^2}\int_M|E_g|^2u^{-1}\ud v_g.
	\end{align*}

	\end{theorem}

	\begin{proof}
		Since $g_0$ is an Einstein metric, a direct computation yields that
		\begin{equation}\label{E_ij}
			u(E_g)_{ij}=-(n-2)\left(\nabla^2_{ij}u-\frac{\Delta_gu}{n}g_{ij}\right)
		\end{equation}
		and
		\begin{equation}\label{Delta_g u}
			\Delta_gu=\frac{n}{2}u^{-1}|\nabla_g u|^2-\frac{R_gu}{2(n-1)}+\frac{u^{-1}R_0}{2(n-1)}.
		\end{equation}
		
		Multiplying  $\nabla^2_{ij}R_g$ on both sides of the equations \eqref{E_ij} and integrating it over $(M,g)$, then there holds
		\begin{equation}\label{E_ijR_ij}
			\int_Mu\langle E_g, \nabla^2 R_g\rangle_g\ud v_g=-(n-2)\int_M\langle \nabla^2u, \nabla^2 R_g\rangle_g\ud v_g+\frac{n-2}{n}\int_M\Delta_g u\cdot\Delta_g R_g\ud v_g
		\end{equation}
		We are going to deal with these three terms one by one.
		
		Firstly, with the help of second Bianchi identity $\nabla_iE_{ij}=\frac{n-2}{2n}\nabla_jR$ and integration by parts, one has
		\begin{equation}\label{term 1}
			\int_Mu\langle E_g, \nabla^2 R_g\rangle_g\ud v_g=-\frac{n-2}{2n}\int_M u|\nabla R_g|_g^2\ud v_g-\int_ME_g(\nabla u,\nabla R_g)\ud v_g.
		\end{equation}
		
		Secondly, integrating by parts, one has
		\begin{align*}
			&\int_M\langle \nabla^2u, \nabla^2 R_g\rangle_g\ud v_g\\
			=&-\int_M\langle \nabla R_g, \nabla \Delta_g u\rangle_g\ud v_g-\int_M Ric_g(\nabla u, \nabla R_g)\ud v_g\\
			=&\int_M\Delta_gR_g\cdot\Delta_g u\ud v_g-\int_ME_g(\nabla u,\nabla R_g)\ud v_g-\frac{1}{n}\int_MR_g\langle \nabla R_g,\nabla u\rangle_g\ud v_g.
		\end{align*}
		Combing these estimates, we can rewrite \eqref{E_ijR_ij} as follows
		\begin{align}
			0=&-\frac{1}{2n}\int_Mu|\nabla R_g|^2\ud v_g-\frac{n-1}{n-2}\int_ME_g(\nabla u,\nabla R_g)\ud v_g\label{equ:2.8}\\
			&+\frac{n-1}{n}\int_M\Delta_g R_g\cdot\Delta_gu\ud v_g-\frac{1}{n}\int_MR_g\langle \nabla R_g,\nabla u\rangle_g\ud v_g\nonumber
		\end{align}
		
		With the help of  the identities \eqref{I_a curvature repre} and \eqref{Delta_g u}, one has
		\begin{align*}
			&\int_M\Delta_g R_g\cdot\Delta_g u\ud v_g\\
			=&\int_M\left(-2(n-1)I_a-\alpha_1|E_g|^2+\alpha_2R_g^2\right)\cdot\Delta_g u\ud v_g\\
			=&2(n-1)\int_M\langle \nabla I_a,\nabla u\rangle_g\ud v_g-\alpha_1\int_M|E_g|^2\left(\frac{n}{2}u^{-1}|\nabla u|^2-\frac{R_gu}{2(n-1)}+\frac{R_0u^{-1}}{2(n-1)}\right)\ud v_g\\
			&-2\alpha_2\int_MR_g\langle \nabla R_g,\nabla u\rangle_g\ud v_g.
		\end{align*}

		Inserting the above formula into \eqref{equ:2.8},  we obtain that
		\begin{align*}
			&2(n-1)^2\int_M\langle \nabla I_a,\nabla u\rangle_g\ud v_g\\
			=&\frac{1}{2}\int_Mu|\nabla R_g|^2\ud v_g+\frac{n(n-1)}{n-2}\int_M E_g(\nabla R_g,\nabla u)\ud v_g\\
			&+\frac{n(n-1)\alpha_1}{2}\int_M|E_g|^2u^{-1}|\nabla u|^2\ud v_g-\frac{\alpha_1}{2}\int_M|E_g|^2R_gu\ud v_g\\
			&+\frac{\alpha_1 R_0}{2}\int_M|E_g|^2u^{-1}\ud v_g\\
			&+\left(2(n-1)\alpha_2+1\right)\int_MR_g\langle\nabla R_g,\nabla u\rangle_g\ud v_g
		\end{align*}

		Multiplying  $ R_gE_{ij}$ on both sides of  \eqref{E_ij} and integrating by parts, we obtain that
		\begin{align*}
			&\int_MuR_g|E_g|^2\ud v_g\\
			=&-(n-2)\int_MR_g\langle E_g,\nabla^2u\rangle_g\ud v_g\\
			=&\frac{(n-2)^2}{2n}\int_MR_g\langle \nabla R_g,\nabla u\rangle_g\ud v_g+(n-2)\int_ME_g(\nabla R_g,\nabla u)\ud v_g
		\end{align*}
		which is equivalent to
		\begin{equation}\label{R_g nabla R_g nabla u}
			\int_MR_g\langle \nabla R_g,\nabla u\rangle_g\ud v_g=\frac{2n}{(n-2)^2}\int_MuR_g|E_g|^2\ud v_g-\frac{2n}{n-2}\int_ME_g(\nabla R_g,\nabla u)\ud v_g
		\end{equation}
		Inserting  the identity \eqref{R_g nabla R_g nabla u} and the notations \eqref{def alpha_1,2} into \eqref{equ:2.8}, there holds
		\begin{align*}
			&2(n-1)^2\int_M\langle \nabla I_a,\nabla u\rangle_g\ud v_g\\
			=&\frac{1}{2}\int_Mu|\nabla R_g|^2\ud v_g+\frac{n-(n-1)(a+4)}{n-2}\int_M E_g(\nabla R_g,\nabla u)\ud v_g\\
			&+\frac{n(n-1)^2(a+4)}{2(n-2)^2}\int_M|E_g|^2u^{-1}|\nabla u|^2\ud v_g\\
			&+\frac{(n-1)(a+4)+2n^2-4n}{2(n-2)^2}\int_Mu|E_g|^2R_g\ud v_g\\
			&+\frac{(n-1)(a+4) R_0}{2(n-2)^2}\int_M|E_g|^2u^{-1}\ud v_g.
		\end{align*}
		
		Thus we finish our proof.
		
	\end{proof}

	With the help of above identity, we are able to establish the following   Case-Gursky-V\'etois  inequalities.
\begin{corollary}\label{cor: rigity inequality about Q}
	Suppose that $(M^n,g_0)$ where $n\geq 3$ is a compact  Einstein manifold with constant scalar curvature $R_0$. Consider a conformal metric
	$g=u^2g_0$ where $u>0$.
	Then  there holds
	$$(n-2)^2(n-1)\int_M\langle\nabla Q_g,\nabla u\rangle_g\ud v_g\geq \frac{n^2-2}{2(n-1)}\int_Mu|E_g|_g^2R_g\ud v_g+R_0\int_M|E_g|_g^2u^{-1}\ud v_g$$
	with the  equality holds if and only if   $g$ is Einstein.
\end{corollary}
\begin{proof}

	Choosing  $a=0$ in Theorem \ref{thm: identity}, one has
	\begin{align}
		&2(n-1)^2\int_M\langle \nabla Q_g,\nabla u\rangle_g\ud v_g\label{equ:3.1}\\
		=&\frac{1}{2}\int_Mu|\nabla R_g|^2\ud v_g+\frac{4-3n}{n-2}\int_M E_g(\nabla R_g,\nabla u)\ud v_g\nonumber\\
		&+\frac{2n(n-1)^2}{(n-2)^2}\int_M|E_g|^2u^{-1}|\nabla u|^2\ud v_g\nonumber\\
		&+\frac{n^2-2}{(n-2)^2}\int_Mu|E_g|^2R_g\ud v_g+\frac{2(n-1)R_0}{(n-2)^2}\int_M|E_g|^2u^{-1}\ud v_g.\nonumber
	\end{align}
	
	With help of Cauchy inequality and Young's inequality, one has
	$$\frac{4-3n}{n-2}E_g(\nabla R_g,\nabla u)\geq-\frac{2n(n-1)^2}{(n-2)^2}|E_g|^2u^{-1}|\nabla u|^2-\frac{(3n-4)^2}{8n(n-1)^2}u|\nabla R_g|^2.$$
	Inserting it into the above identity \eqref{equ:3.1}, one has
	\begin{align}
		&2(n-1)^2\int_M\langle \nabla Q_g,\nabla u\rangle_g\ud v_g\label{equ:3.2}\\
		\geq &\frac{4(n-4)(n-1)^2+7n^2-8n}{8n(n-1)^2}\int_Mu|\nabla R_g|^2\ud v_g\nonumber\\
		&+\frac{n^2-2}{(n-2)^2}\int_Mu|E_g|^2R_g\ud v_g+\frac{2(n-1)R_0}{(n-2)^2}\int_M|E_g|^2u^{-1}\ud v_g.\nonumber
	\end{align}
	
	Immediately, one has
	\begin{equation}\label{inequality for Q}
		(n-2)^2(n-1)	\int_M\langle\nabla Q_g,\nabla u\rangle_g\ud v_g\geq \frac{n^2-2}{2(n-1)}\int_Mu|E_g|_g^2R_g\ud v_g+R_0\int_M|E_g|_g^2u^{-1}\ud v_g.
	\end{equation}
	On one hand, if the equality holds in \eqref{inequality for Q},  with the  help of the inequality \eqref{equ:3.2}, one must have $|\nabla R_g|=0$ which means that the scalar curvature  $R_g$  is a constant. Then Obata theorem shows that $g$ is Einstein. On the other hand, if $g$ is Einstein, it is not  hard to check that the equality is achieved by \eqref{Einstein Q_0}.
	
	Thus we finish our proof.
	
	\end{proof}

	To slightly extend Case's Theorem \ref{theorem case}, we need the following well known lemma. For reader's convenience, we sketch the proof here.
	\begin{lemma}\label{lem: A}
		Let $A$ be a $n\times n$  symmetric  matrix satisfying $trace(A)=0$. Let $x$ and $y$ be two $n\times 1$ vector. There holds
		$$|x^TAy|\leq \sqrt{\frac{n-1}{n}}|A|\cdot |x|\cdot |y|.$$
	\end{lemma}
	\begin{proof}
		With the  help of orthogonality, one may assume that $A$ is a diagonal matrix and let $\lambda_i$ be the diagonal elements where $1\leq i\leq n$ and $|\lambda_i|\leq |\lambda_1|$ for all $1\leq i\leq n$. Notice that
		\begin{equation}\label{|A|leq n/n-1lambda_1}
			|A|^2=\sum^n_i\lambda_i\geq \lambda_1^2+\frac{1}{n-1}\left(\sum_{i=2}^n\lambda_i\right)^2=\frac{n}{n-1}\lambda_1^2.
		\end{equation}
		With the help of \eqref{|A|leq n/n-1lambda_1} and Cauchy inequality, one has
		$$|x^TAy|=|\sum_i\lambda_ix_iy_i|\leq |\lambda_1|\cdot|x|\cdot|y|\leq \sqrt{\frac{n-1}{n}}|A|\cdot |x|\cdot|y|.$$
	\end{proof}

		\begin{corollary}\label{cor: rigity inequality}
		Suppose that $(M^n,g_0)$ where $n\geq 3$ is a compact  Einstein manifolds with constant scalar curvature $R_0$ and $n\geq 3$. Consider the conformal metric
		$g=u^2g_0$ where $u>0$.  If the constant $a$ satisfies
		\begin{equation}\label{a range}
			\frac{(n-2)(n-4)-n\sqrt{n^2+4n}}{2(n-1)}\leq a\leq \frac{(n-2)(n-4)+n\sqrt{n^2+4n}}{2(n-1)},
		\end{equation}
		then  there holds
		$$\int_M\langle\nabla I_a(g),\nabla u\rangle_g\ud v_g\geq a_1\int_Mu|E_g|^2R_g\ud v_g+a_2R_0\int_M|E_g|^2u^{-1}\ud v_g.$$
		where $a_1(n)=\frac{(n-1)(a+4)+2n^2-4n}{4(n-2)^2(n-1)^2}$
		and $a_2=\frac{a+4}{4(n-2)^2(n-1)}$.
	\end{corollary}
	\begin{proof}
	With the help of Lemma \ref{lem: A} and Cauchy-Schwarz inequality as well as Young's inequality, there holds
	\begin{align*}
		&\frac{n-(n-1)(a+4)}{n-2}\int_M E_g(\nabla R_g,\nabla u)\ud v_g\\
		\geq &-|\frac{n-(n-1)(a+4)}{n-2}|\sqrt{\frac{n-1}{n}}\int_M|E_g|\cdot |\nabla R_g|\cdot |\nabla u|\ud v_g\\
		\geq &-\frac{1}{2}\int_Mu|\nabla R_g|^2\ud v_g-\frac{\left(n-(n-1)(a+4)\right)^2(n-1)}{2(n-2)^2n}\int_M|E_g|^2u^{-1}|\nabla u|^2.
	\end{align*}
	Using the condition \eqref{a range} and  Theorem \ref{thm: identity},  we obtain our desired result.
		\end{proof}

For reader's convenience, we repeat the proof of  V\'etois'  Theorem \ref{theorem of Vetois} and Case's Theorem \ref{theorem case} via the  Case-Gursky-V\'etois   identity.

\vspace{3em}
	\noindent
	{\bf Proof of Theorem \ref{theorem of Vetois}:}
	
		When the scalar curvature $R_0=0$, then Q-curvature  of $g_0$ also vanishes by \eqref{Einstein Q_0}. For $n\geq 3$ and $n\not=4$,  if the  Q-curvature of conformal metric of $g=\varphi^{\frac{4}{n-4}}g_0$ is constant,  the equation \eqref{Einstein paneitz} yields that
	\begin{equation}\label{null Q for n=3,5}
		\Delta_{g_0}^2\varphi=C\varphi^{\frac{n+4}{n-4}}
	\end{equation}
	where $C$ is a constant. By integrating \eqref{null Q for n=3,5} over $(M,g_0)$, it is easy to say that the constant $C$ must be zero. Immediately, $\Delta_{g_0}\varphi$ must be a constant and such constant must be zero by integrating it over $(M,g_0)$ again. Then  $\varphi$ must be a constant. For $n=4$, consider $g=e^{2\varphi}g_0$ and the  same argument yields that $\varphi$ must be a constant.

	When $R_0>0$, with the  help of \eqref{def of Q-curvature}, one has
	$$Q_{g_0}=\frac{n^2-4}{8n(n-1)^2}R_0^2>0.$$ If the Q-curvature $Q_g$ of the conformal metric $g=u^2g_0$ is constant. By integrating the equations \eqref{conformal Q-curvature n not 4} and \eqref{conformal Q-curvature for n=4} over $(M,g_0)$, it is easy to see that the constant $Q_g>0.$ Apply Lemma \ref{GM lemma}, the scalar curvature $R_g$ is positive. Making use of Corollary \ref{cor: rigity inequality about Q}, one has
	$$0 \geq \frac{n^2-2}{2(n-1)}\int_Mu|E_g|_g^2R_g\ud v_g+R_0\int_M|E_g|_g^2u^{-1}\ud v_g.$$
	Since $u,R_g,R_0$ are all positive, one must have  $E_g\equiv0$ which means that $g$ is Einstein. If Furthermore,  if    $(M,g_0)$ is not conformally equivalent to round sphere, Obata theorem yields that  $u$ must be a positive constant.

	\vspace{3em}
	\noindent
	{\bf Proof of Theorem \ref{theorem case}:}
	
		If $R_0=0$, one has  $-\frac{4(n-1)}{n-2}\Delta_0u^{\frac{n-2}{2}}=R_gu^{\frac{n+2}{2}}\geq 0$. Integrating it over $(M^n,g_0)$, the  left side vanishes  and  then we must have $R_g\equiv 0.$ Immediately, $u$ must be a  constant. Otherwise, $R_0>0$, Corollary  \ref{cor: rigity inequality} yields that $E_g=0$ which means that $g$ is Einstein.

	\section{Applications}\label{sec:4}

	With the help of a continuity method and the  maximum principle, Gursky and Malchiodi \cite{G-M} establish an important theorem related to the positivity of Q-curvature and scalar curvature in conformal classes of the metrics.
	\begin{lemma}\label{GM lemma}(See Theorem 2.2 in \cite{G-M}, Proposition 2.1 in \cite{Hang-Yang},Theorem 2.3 in  \cite{Vetois})
		Given a compact manifold $(M^n,g)$ with $n\geq 3$ with positive scalar curvature  and non-negative Q-curvature. Consider the conformal metric $\tilde g=u^2g$. If the Q-curvature $Q_{\tilde g}\geq 0$, then  the scalar curvature $R_{\tilde g}$ is positive.
	\end{lemma}

	\vspace{3em}
	\noindent
	{\bf Proof of Theorem \ref{thm: neq 5}:}
	
	For $n\geq 5$,	consider the conformal metric $g=\varphi^{\frac{4}{n-4}}g_0$ and then $$Q_g=\frac{2}{n-4}\left(\epsilon \varphi^{-\frac{8}{n-4}}+\varphi^{p-\frac{n+4}{n-4}}\right)>0.$$ Lemma \ref{GM lemma}  yields that $R_g>0$.
	
	Choosing  $u=\varphi^{\frac{2}{n-4}}$ in Corollary  \ref{cor: rigity inequality about Q}, one has
	\begin{align*}
		\langle\nabla Q_g,\nabla u\rangle_g=&-\frac{32\varepsilon}{(n-4)^3}\varphi^{-\frac{6}{n-4}-2}|\nabla \varphi|_g^2-\frac{4}{(n-4)^2}(\frac{n+4}{n-4}-p)\varphi^{p-\frac{n+2}{n-4}-2}|\nabla\varphi|_g^2\\
		=&C(\varphi, \varepsilon, p,n)|\nabla\varphi|^2_g
	\end{align*}
	where
	$$C(\varphi, \varepsilon, p,n)=-\frac{32\varepsilon}{(n-4)^3}\varphi^{-\frac{6}{n-4}-2}-\frac{4}{(n-4)^2}(\frac{n+4}{n-4}-p)\varphi^{p-\frac{n+2}{n-4}-2}.$$
	
	By using Corollary  \ref{cor: rigity inequality about Q} and the facts the scalar curvatures $R_g$ and $R_0$ are both positive, we know that
	\begin{equation}\label{Cvarphi nabla varphi^2 geq 0}
		\int_MC(\varphi, \varepsilon, p,n)|\nabla\varphi|^2_g\ud v_g\geq 0.
	\end{equation}
	If $\epsilon+\frac{n+4}{n-4}-p>0$,  it is easy to check  that $C(\varphi, \varepsilon, p,n)<0$ based on our assumptions. Immediately, by using \eqref{Cvarphi nabla varphi^2 geq 0}, one has  $|\nabla \varphi|_g=0$ which means that $\varphi$ must be a constant. If $\epsilon=0$ and $p=\frac{n+4}{n-4}$, due to our assumption, $(M^n,g_0)$ is not conformally equivalent to a round sphere.  In this case, one has
	$$C(\varphi, \varepsilon, p,n)=0,$$
	Corollary  \ref{cor: rigity inequality about Q} yields that
	$$0\geq \int_Mu|E_g|_g^2R_g\ud v_g+R_0\int_M|E_g|_g^2u^{-1}\ud v_g.$$
	Noticing that $R_g>0$ and $R_0>0$, one has	$E_g\equiv 0$ which yields that $g$  has constant scalar curvature.  Immediately,  Obata theorem deuces that $\varphi$ must be a constant.

	\vspace{2em}
	\noindent
	{\bf Proof of Theorem \ref{thm: n=4}:}
	For $n=4$,
	consider the conformal metric $g=e^{2\varphi}g_0$ and Q-curvature satisfies
	$$Q_g=\varepsilon e^{-4\varphi}+e^{(p-4)\varphi}>0.$$  Then Lemma \ref{GM lemma} yields that $R_g>0$.
	Then choosing $u=e^{\varphi}$, one may easily check that
	$$\langle \nabla Q_g,\nabla u\rangle_g=-4\varepsilon e^{-3\varphi}|\nabla\varphi|^2_g-(4-p)e^{(p-3)\varphi}|\nabla\varphi|^2_g\leq 0.$$
	
	Similarly  as the proof of Theorem \ref{thm: neq 5}, $\varphi$ must be a constant.

	\vspace{2em}
	\noindent
	{\bf Proof of Theorem \ref{thm: n=3}:}
	
	For $n=3$, consider the conformal metric $g=\varphi^{-7}g_0$ and Q-curvature satisfies
	$$Q_g=2\varepsilon\varphi^8+2\varphi^{p+7}>0$$
	which yields that $R_g>0$ by Lemma \ref{GM lemma}.
	Choose $u=\varphi^{-2}$ and then a direct computation yields that
	$$\langle\nabla Q_g,\nabla u\rangle_g=-32\varepsilon\varphi^4|\nabla\varphi|^2_g-4(p+7)\varphi^{p+3}|\nabla\varphi|_g^2\leq 0.$$
	
	Similarly  as the proof of Theorem \ref{thm: neq 5}, $\varphi$ must be a constant.

	\vspace{2em}
	\noindent
	{\bf Proof of Theorem \ref{thm: f(u) n not 4}:}
	
	For $n\geq 5$, our assumption $f(t)\geq 0$ yields that
	$$P_{g_0}\varphi\geq 0.$$
	Then the strong maximum principle of Theorem 2.2 in \cite{G-M} shows that $\varphi\equiv 0$ or $\varphi>0$. Thus we only need to deal with $\varphi>0$. In this situation, consider the conformal metric $g=\varphi^{\frac{4}{n-4}}g_0$. Then the Q-curvature satisfies
	$$Q_g=\varphi^{-\frac{n+4}{n-4}}f(\varphi)\geq 0.$$
	Making use of this fact, Lemma \ref{GM lemma} shows that $R_g>0.$ Then choosing $u=\varphi^{\frac{2}{n-4}}$,  one has
	\begin{equation}\label{nabla Q_g nabla u}
		\langle\nabla Q_g, \nabla u\rangle_g=\frac{4}{n-4}\varphi^{\frac{4}{n-4}-1}F'(\varphi)|\nabla \varphi|^2_g\leq 0
	\end{equation}
	where
	$F(t)=t^{-\frac{n+4}{n-4}}f(t)$.  Corollary  \ref{cor: rigity inequality about Q} and the inequality  yield that
	$$0\geq \frac{n^2-2}{2(n-1)}\int_Mu|E_g|_g^2R_g\ud v_g+R_0\int_M|E_g|_g^2u^{-1}\ud v_g.$$
	Since $R_g>0$ and $R_0>0$, one must have $E_g=0$ which means that $R_g$ is a constant. On one hand, if$(M^n,g_0)$ is not conformally equivalent to the round sphere, Obata theorem implies that $\varphi$ must be a constant.  On the other hand, if$(M^n,g_0)$ is conformally equivalent to the round sphere, Corollary \ref{cor: rigity inequality about Q} and the  \eqref{nabla Q_g nabla u} show that
	$$\frac{4}{n-4}\int_M\varphi^{\frac{4}{n-4}-1}F'(\varphi)|\nabla \varphi|^2_g\ud v_g\geq 0.$$ Since $F'(t)<0 $ for all $t>0$, one must have $|\nabla\varphi|_g\equiv 0$  which means that $\varphi$ is a constant.
	
	For $n=3$, one has
	$$P_{g_0}\varphi\leq 0.$$
	Proposition 2.1 in \cite{Hang-Yang} shows that $\varphi\equiv 0$ or $\varphi>0.$ Since the left of the proof is the same as before, we omit the details.
	
	Thus we finish our proof.
	
	\vspace{2em}
	\noindent
	{\bf Proof of Theorem \ref{thm: f(u) n =4}:}
	
	Firstly,  consider the conformal metric $g=e^{2\varphi}g_0$ and then
	$$Q_g=e^{-4\varphi}f(e^{\varphi})\geq 0.$$
	Immediately, Lemma \ref{GM lemma} yields that $R_g>0.$  The remain is the same as the proof of Theorem \ref{thm: f(u) n not 4} and we omit it.


\begin{thebibliography}{99}
		\bibitem{BVV}
		M. Bidaut-Veron, L.  Veron,  Nonlinear elliptic equations on compact Riemannian mani- folds and asymptotics of Emden equations. Inventiones Math., 1991, 106: 489-539.
		\bibitem{Branson}
		T. Branson,		Differential operators canonically associated to a conformal structure,
		Math. Scand.57(1985), no.2, 293–345.
		\bibitem{Brezis-Li}
		H. Brezis,  Y.  Li, Some nonlinear elliptic equations have only constant solutions,  J. Partial Differential Equations 19, no. 3 (2006): 208–17.
		\bibitem{Case}
		J. Case,  The Obata–Vétois argument and its applications,  Journal für die reine und angewandte Mathematik (Crelle's Journal), vol. 2024, no. 815, 2024, pp. 23-40.
		\bibitem{Chang-Yang 95 Ann}
		S. 	Chang,  P. Yang, Extremal metrics of zeta function determinants on 4-manifolds,
		Ann. of Math. (2)142(1995), no.1, 171–212.
		\bibitem{Dj-Ma 08 Ann}
		Z. 	Djadli, A.  Malchiodi, Existence of conformal metrics with constant Q-curvature,
		Ann. of Math. (2)168(2008), no.3, 813–858.
		\bibitem{Gidas-Sprcuk CPAM}
		B. 	Gidas,  J.  Spruck, Global and local behavior of positive solutions of nonlinear elliptic equations. Comm. Pure Appl. Math., 1981, 34: 525–598.
		
		\bibitem{Gover}
		A. 	Gover, Laplacian operators and Q-curvature on conformally Einstein manifolds,  Math. Ann. 336, No. 2, 311-334 (2006).
		\bibitem{Gur 97 CMP}
		M. Gursky, Uniqueness of the functional determinant, Comm. Math. Phy. 189(1997), no.3, 655-665.
		\bibitem{G-M}
		M. 	Gursky, A. Malchiodi,
		A strong maximum principle for the Paneitz operator and a non-local flow for the Q-curvature,
		J. Eur. Math. Soc. (JEMS)17(2015), no.9, 2137–2173.
		\bibitem{Hang-Yang 20 IMRN}
		F. 	Hang, P. Yang, 		A perturbation approach for Paneitz energy on standard three sphere,
		Int. Math. Res. Not. IMRN(2020), no.11, 3295–3317.
		\bibitem{Hang-Yang CPAM for ngeq 5}
		F. 	Hang, P. Yang, 	Q-curvature on a class of manifolds with dimension at least 5,
		Comm. Pure Appl. Math.69(2016), no.8, 1452–1491.
		\bibitem{Hang-Yang}
		F. 	Hang, P. Yang, 	Q curvature on a class of 3-manifolds,
		Comm. Pure Appl. Math.69(2016), no.4, 734–744.
		\bibitem{Hyder-Ngo}
		A. Hyder, Q. Ng\^o,  On the Hang-Yang conjecture for GJMS equations on $\mathbb{S}^n$,
		Math. Ann.389(2024), no.3, 2519–2560.
		\bibitem{Obata 1}
		M. Obata, Certain conditions for a Riemannian manifold to be isometric with a sphere, J. Math. Soc. Japan 14 (1962), 333–340.
		\bibitem{Obata 2}
		M. Obata, The conjectures on conformal transformations of Riemannian manifolds, J. Differential Geom. 6 (1971/72), 247–258.
		\bibitem{Paneitz}
		S. Paneitz, A quartic conformally covariant differential operator for arbitrary pseudo-Riemannian manifolds,
		SIGMA Symmetry Integrability Geom. Methods Appl.4(2008), Paper 036, 3 pp.
		\bibitem{Vetois}
		J. V\'etois, Uniqueness of conformal metrics with constant Q-curvature on closed Einstein manifolds, Potential Anal. (2023), 10.1007/s11118-023-10117-1.
		\bibitem{Zhang SH}
		S. Zhang, A Liouville-type theorem of the linearly perturbed Paneitz equation on $S^3$,
		Int. Math. Res. Not. IMRN(2023), no.2, 1730–1759.
	\end{thebibliography}
\end{document}